\documentclass[12pt,reqno,a4paper]{article}
\usepackage[margin=1in]{geometry}
\usepackage{amsmath,amssymb,amsthm,mathtools}
\usepackage{enumitem}
\usepackage[colorlinks=true,linkcolor=blue,citecolor=blue,urlcolor=blue]{hyperref}
\usepackage{authblk}
\usepackage{orcidlink}

\newcommand{\Z}{\mathbb{Z}}
\newcommand{\Zi}{\Z[i]}
\newcommand{\Rbullet}{\Zi\setminus\{0\}} 
\newcommand{\Units}{\mathcal{U}}
\newcommand{\Supp}{\operatorname{Supp}}

\newcommand{\Pe}{\mathcal{P}}

\newcommand\blfootnote[1]{%
  \begingroup
  \renewcommand\thefootnote{}\footnote{\hspace{-1.8em}#1}%
  \addtocounter{footnote}{-1}%
  \endgroup
}

\theoremstyle{plain}
\newtheorem{theorem}{Theorem}[section]
\newtheorem{proposition}[theorem]{Proposition}
\newtheorem{lemma}[theorem]{Lemma}
\newtheorem{corollary}[theorem]{Corollary}
\newtheorem*{problem}{Problem}

\theoremstyle{definition}
\newtheorem{definition}[theorem]{Definition}
\newtheorem{example}[theorem]{Example}

\theoremstyle{remark}
\newtheorem{remark}[theorem]{Remark}

\title{{A Coprimality Topology on the Gaussian Integers: Kolmogorov Quotient and Gaussian Prime Density}}

\author{
    Souvik Mandal$^{\,a}$ 
    \\
    \footnotesize\itshape $^{a}$Department of Mathematics, Indian Institute of Technology Madras, \\ 
    \footnotesize\itshape Chennai-600036, Tamil Nadu, India\\
    }

\date{}
\begin{document}
\maketitle
\blfootnote{$^{*}$\textit{Email addresses:} \texttt{ma22d014@smail.iitm.ac.in, ssouvik.xyz@gmail.com.}}
 \vspace{-3.5em}
\begin{abstract}
\noindent This article investigates the coprimality topology on the set of non-zero Gaussian integers, $\mathbb{Z}[i]\setminus\{0\}$, generated by the arithmetic basis $\sigma_\alpha=\{\beta\neq 0 : \gcd(\alpha,\beta)\sim 1\}$. By analyzing prime supports up to associates, we establish that two points are topologically indistinguishable if and only if their supports coincide. We demonstrate that the space is both hyperconnected and ultraconnected, and we explicitly characterize its Kolmogorov quotient $X$ as the space of finite subsets of associate classes of Gaussian primes, endowed with the basis $\mathcal{O}_F=\{S\in X:S\cap F=\emptyset\}$ for $F\in X$. Finally, we demonstrate that the set of Gaussian primes is dense in $\mathbb{Z}[i]\setminus\{0\}$ with respect to the coprimality topology, thereby establishing a topological proof of the infinitude of Gaussian primes.
\end{abstract}

\vspace{-0.5em}
\begin{center}
\begin{minipage}{0.845\textwidth}
    \footnotesize
    \begin{list}{}{%
        \leftmargin=4.3em 
        \labelwidth=5em
        \labelsep=0pt \parsep=0pt \topsep=0pt \itemsep=0pt
    }
        \item[\textit{Keywords:}\hfill]Gaussian integers, coprimality topology, Hyperconnectedness, Ultraconnectedness, Kolmogorov quotient,
Gaussian primes, Density.
    \end{list}

    \vspace{2pt} 

    \begin{list}{}{%
        \leftmargin=17.8em 
        \labelwidth=18.5em
        \labelsep=0pt \parsep=0pt \topsep=0pt \itemsep=0pt
    }
        \item[2020\hspace{1mm}\textit{Mathematics Subject Classification:}\hfill] Primary 11R04, 54A05;\\ Secondary 11A41, 13G05, 54B15, 54D05, 54D10.
    \end{list}
\end{minipage}
\end{center}
\vspace{1em}
\section{Introduction}
The intersection of topology and number theory has a rich history, most famously inaugurated by Furstenberg's 1955 topological proof of the infinitude of prime numbers \cite{Furstenberg}. Furstenberg's topology on $\mathbb{Z}$, generated by arithmetic progressions, is Hausdorff and metrizable. In contrast, Golomb \cite{Golomb} and Kirch \cite{Kirch} explored topologies on $\mathbb{N}$ that are connected but not Hausdorff, emphasising the arithmetic nature of open sets.

In 2024, Mac\'ias introduced a coarser topology on $\mathbb{N}$ generated by sets of integers coprime to a fixed integer \cite{MaciasIntegers}. This construction was subsequently generalized to any integral domain $R$, yielding what we refer to as the coprimality topology \cite{MaciasDomains}. Within the general theory of integral domains (in particular for infinite PIDs) the coprimality topology is known to be connected and rarely satisfies higher separation axioms.

However, moving from the general theory of integral domains to the specific ring of Gaussian integers $\Zi$ offers structural insights. It is an infinite PID (hence a UFD) with a \emph{nontrivial finite} unit group.
While \cite{MaciasDomains} establishes the topological framework for general integral domains, the arithmetic properties of $\mathbb{Z}[i]$ allow us to replace abstract existential results with explicit, constructive proofs. We exploit the finite unit group and Euclidean structure of the Gaussian integers to sharpen the general theory, contributing to the literature by:

\begin{enumerate}[label=\textup{(\roman*)}]
\item A \emph{prime support} description of basic open sets and closures in $\Rbullet$;
\item A characterization of topological indistinguishability in terms of
associate classes of Gaussian primes;
\item An explicit description of the Kolmogorov ($T_{0}$) quotient of $\Rbullet$ as the space of finite subsets of associate classes of Gaussian primes endowed with a specific basis;
\item A topological proof of the density of Gaussian primes in $\Rbullet$. In contrast to the approach in \cite{MaciasDomains}, which relies on the existence of infinitely many maximal ideals to establish the density and infinitude of primes, we prove the density of Gaussian primes directly using the arithmetic properties of $\mathbb{Z}[i]$, thereby deriving their infinitude as a consequence.
\end{enumerate}

\section{The Coprimality Topology on \texorpdfstring{$\Rbullet$}{}}

\hspace{6mm}Let $\Zi=\{a+bi: a,b\in\Z\}$ be the ring of Gaussian integers equipped with the standard multiplicative norm $N(a+bi)=a^2+b^2$. The group of units is finite, given by $\Units=\{1,-1,i,-i\}$.

Since $\Zi$ is a Euclidean domain (and thus a UFD), every non-zero non-unit Gaussian integer admits a unique factorization into Gaussian primes up to reordering and unit multiples. Two elements $\alpha, \beta \in \Rbullet$ are \emph{associates}, denoted $\alpha \sim \beta$, if $\alpha = u\beta$ for some $u \in \Units$. We write $\gcd(\alpha,\beta)\sim 1$ to indicate that $\alpha$ and $\beta$ are coprime.

To connect the arithmetic structure with the topological construction, we introduce the notion of \emph{prime support}.
\begin{definition}
For any $\alpha \in \Rbullet$, the \emph{prime support}, denoted by $\Supp(\alpha)$, is the finite set of associate classes of Gaussian primes dividing $\alpha$. In particular, if $\alpha$ is a unit, $\Supp(\alpha) = \emptyset$.
\end{definition}

We now define the topology of our interest. For each $\alpha\in\Rbullet$, consider the \emph{basic set}
\[
\sigma_\alpha := \{\beta\in\Rbullet : \gcd(\alpha,\beta)\sim 1\}.
\]
In terms of prime support, $\beta \in \sigma_\alpha$ if and only if $\Supp(\alpha) \cap \Supp(\beta) = \emptyset$.

\begin{definition}
The \emph{coprimality topology} $\mathcal{T}$ on $\Rbullet$ is the topology generated by the collection $\mathcal{B} = \{ \sigma_\alpha : \alpha \in \Rbullet \}$. We refer to the pair $(\Rbullet,\mathcal{T})$ as the \emph{coprimality space} of $\Zi$.
\end{definition}

\begin{proposition}\label{prop:basis}
The collection $\mathcal{B}$ forms a basis for $\mathcal{T}$. Moreover, for all $\alpha,\beta\in\Rbullet$,
\[
\sigma_\alpha\cap\sigma_\beta=\sigma_{\alpha\beta}.
\]
\end{proposition}

\begin{proof}
First, $\sigma_1 = \Rbullet$ implies that $\mathcal{B}$ covers the space.
For the intersection property, let $\gamma \in \sigma_\alpha \cap \sigma_\beta$. By definition, $\gcd(\gamma, \alpha) \sim 1$ and $\gcd(\gamma, \beta) \sim 1$. Since $\Zi$ is a UFD, this implies $\gcd(\gamma, \alpha\beta) \sim 1$, so $\gamma \in \sigma_{\alpha\beta}$.

Conversely, if $\gamma \in \sigma_{\alpha\beta}$, then no prime factor of $\gamma$ divides $\alpha\beta$. Consequently, $\gamma$ shares no prime factors with $\alpha$ nor with $\beta$, implying $\gamma \in \sigma_\alpha \cap \sigma_\beta$.
\end{proof}

\begin{remark}\label{rem:support-dependence}
The basic open sets depend solely on the prime factors of the index. Specifically, $\sigma_\alpha = \sigma_\beta$ if and only if $\Supp(\alpha) = \Supp(\beta)$. This observation is crucial because each Gaussian prime has four associates, yet the topology is insensitive to this choice of representative.
\end{remark}

\begin{example}
Consider the Gaussian prime $1+i$, which has norm $2$. A Gaussian integer $a+bi$ is divisible by $1+i$ if and only if $a\equiv b\pmod 2$. Thus, the basic open set
\[
\sigma_{1+i}=\{a+bi\in\Rbullet : a\not\equiv b\pmod 2\}
\]
consists of all Gaussian integers with mixed parity. Furthermore, since $2 \sim (1+i)^2$, we have $\Supp(2)=\Supp(1+i)$, and therefore $\sigma_2=\sigma_{1+i}$.
\end{example}

\section{Separation and Connectivity Properties}
In this section, we characterize topological indistinguishability via prime support, which implies that the coprimality space of $\Zi$ fails to satisfy the $T_{0}$ axiom. Subsequently, we establish that the space is hyperconnected and ultraconnected.
\subsection{Separation and Arithmetic Characterization of Indistinguishability}

We recall the following foundational definitions,

\begin{definition}
Let $(X, \tau)$ be a topological space.
\begin{itemize}
    \item Two points $x, y \in X$ are termed \emph{topologically indistinguishable} if they belong to precisely the same collection of open sets; that is, for every $U \in \tau$, $x \in U$ holds if and only if $y \in U$.
    \item The space $X$ is said to satisfy the \emph{$T_0$ separation axiom} (or is classified as a \emph{Kolmogorov space}) if no two distinct points are topologically indistinguishable. In other words, the space $X$ is said to be a \emph{Kolmogorov space} if, for any pair of distinct points $x, y \in X$, there exists an open set $U$  of $X$ such that \emph{either} $x \in U$ and $y \notin U$ \emph{or} $y \in U$ and $x \notin U$.
\end{itemize}
\end{definition}

The following theorem establishes interplay between \emph{indistinguishability} and \emph{prime support}. 
\begin{theorem}\label{thm:indistinguishable}
Two elements $x, y \in \Rbullet$ are topologically indistinguishable if and only if $\Supp(x) = \Supp(y)$.
\end{theorem}

By Remark~\ref{rem:support-dependence}, for any $\alpha \in \Rbullet$, we have $x \in \sigma_\alpha$ \emph{if and only if}  $\Supp(x) \cap \Supp(\alpha) = \emptyset$; an identical characterization applies to $y$ as well. Consequently, the equality $\Supp(x) = \Supp(y)$ implies that $x$ and $y$ belong to precisely the same collection of basic open sets and, by extension, to the same open sets.

Conversely, suppose $\Supp(x) \neq \Supp(y)$, and let $[\pi]$ be a Gaussian prime class in the symmetric difference of their supports. Assume, without loss of generality, that $[\pi] \in \Supp(x) \setminus \Supp(y)$. It follows that $y \in \sigma_\pi$ while $x \notin \sigma_\pi$, which establishes that $x$ and $y$ are topologically distinguishable.

\begin{corollary}
All four units  $\Units=\{1,-1,i,-i\}$ are topologically indistinguishable. Hence, The space $(\Rbullet,\mathcal{T})$ is not a \emph{Kolmogorov $(T_{0})$ space}.
\end{corollary}

In addition to Theorem~\ref{thm:indistinguishable}, we establish an explicit characterization for the closures of elements in $\Rbullet$.
\begin{proposition}\label{prop:closure-support}
For $x,y\in\Rbullet$,
\[
y\in \overline{\{x\}}
\quad\Longleftrightarrow\quad
\Supp(x)\subseteq \Supp(y).
\]
\end{proposition}

\begin{proof}
By definition, $y \in \overline{\{x\}}$ if and only if every open neighbourhood of $y$ contains $x$. Since $\mathcal{B}$ constitutes a basis for the topology $\mathcal{T}$, this condition is equivalent to the requirement that for every $\alpha \in \mathbb{Z}[i] \setminus \{0\}$, we must have $y \in \sigma_\alpha$ implies $x \in \sigma_\alpha$.

Following Remark~\ref{rem:support-dependence}, the condition $y \in \sigma_\alpha$ is characterized by the disjointness of their respective supports, namely $\Supp(y) \cap \Supp(\alpha) = \emptyset$. Suppose, on the contrary, assume that $\Supp(x) \not\subseteq \Supp(y)$. We may then choose an associate class of a Gaussian prime $[\pi]$ such that $[\pi] \in \Supp(x) \setminus \Supp(y)$. By setting $\alpha = \pi$, it follows that $y \in \sigma_\pi$ while $x \notin \sigma_\pi$, which contradicts the assumption that $y \in \overline{\{x\}}$.

Conversely, assume that $\Supp(x) \subseteq \Supp(y)$. If $y \in \sigma_\alpha$ for some $\alpha \in \mathbb{Z}[i] \setminus \{0\}$, then the condition $\Supp(y) \cap \Supp(\alpha) = \emptyset$ necessarily implies that $\Supp(x) \cap \Supp(\alpha) = \emptyset$. Consequently, $x \in \sigma_\alpha$ holds for all such $\alpha$, thereby establishing the desired inclusion.
\end{proof}
\begin{remark}\label{rem:specialization}
Proposition~\ref{prop:closure-support} demonstrates that the specialization preorder on $(\mathbb{Z}[i]\setminus\{0\},\mathcal{T})$ is determined by the inclusion of prime supports. Thus, points with larger prime support are topologically more specialized. In particular, the addition of prime factors restricts the collection of basic open neighbourhoods containing a point.
\end{remark}

\subsection{Hyperconnectedness and Ultraconnectedness}

In preparation for establishing the strong connectivity properties of the \emph{coprimality space} of $\Zi$, we recall the following definitions.
\begin{definition}
A topological space $(X, \tau)$ is said to be,
\begin{enumerate}
    \item \emph{Hyperconnected} (or \emph{irreducible}) if every pair of nonempty open sets has a nonempty intersection; that is, for all $U, V \in \tau \setminus \{\emptyset\}$, we have $U \cap V \neq \emptyset$.
    \item \emph{Ultraconnected} if every pair of nonempty closed sets has a nonempty intersection; that is, for any two closed sets $F, G \subseteq X$ such that $F, G \neq \emptyset$, it follows that $F \cap G \neq \emptyset$.
\end{enumerate}
\end{definition}
The following lemma asserts that the units of $\Zi$ belong to every nonempty open set of the \emph{coprimality space} of $\Zi$.
\begin{lemma}\label{lem:units-in-open}
Every nonempty open subset $U$ of $\mathbb{Z}[i] \setminus \{0\}$ contains the group of units $\mathcal{U}$.
\end{lemma}

\begin{proof}
By the definition of a basis, the nonempty open set $U$ must contain a basic open set $\sigma_{\alpha}$ for some $\alpha \in \mathbb{Z}[i] \setminus \{0\}$. Since every unit $u \in \mathcal{U}$ is coprime to all elements in $\mathbb{Z}[i] \setminus \{0\}$, the condition $\gcd(u, \alpha) \sim 1$ is trivially satisfied for any such $\alpha$. It follows that $u \in \sigma_{\alpha} \subseteq U$, which establishes the required inclusion that $\mathcal{U} \subseteq U$.
\end{proof}

We now establish the strong connectivity properties of the \emph{coprimality space} of $\Zi$.
\begin{theorem}\label{thm:hyper-ultra}
The space $(\Rbullet,\mathcal{T})$ is both hyperconnected and ultraconnected.
\end{theorem}

\begin{proof}
Let $U$ and $V$ be any two nonempty open subsets of $\mathbb{Z}[i] \setminus \{0\}$. According to Lemma~\ref{lem:units-in-open}, the group of units $\mathcal{U}$ is contained in every nonempty open set; consequently, $\mathcal{U} \subseteq U \cap V$. Since $\mathcal{U}$ is nonempty, the intersection $U \cap V$ is necessarily nonempty, which satisfies the condition for hyperconnectedness.

 Let $F$ and $G$ denote nonempty closed subsets of $\Rbullet$ and choose elements $x \in F$ and $y \in G$. By the properties of closures, the inclusions $\overline{\{x\}} \subseteq F$ and $\overline{\{y\}} \subseteq G$ hold. We consider the product $xy \in \mathbb{Z}[i] \setminus \{0\}$. In view of Proposition~\ref{prop:closure-support}, the prime support of a product contains the supports of its factors, namely $\Supp(x) \subseteq \Supp(xy)$ and $\Supp(y) \subseteq \Supp(xy)$. These inclusions imply that $xy \in \overline{\{x\}}$ and $xy \in \overline{\{y\}}$. It follows that $xy \in F \cap G$, thereby demonstrating that $F \cap G \neq \emptyset$.
\end{proof}

\section{A Description of the Kolmogorov Quotient}

The coprimality space of $\Zi$ fails to satisfy the $T_{0}$ separation axiom. So we consider its Kolmogorov quotient. Recall that for any topological space $(Y, \tau)$, the \emph{Kolmogorov quotient} is the $T_{0}$ space obtained by identifying topologically indistinguishable points; specifically, it is the quotient space $Y/{\sim}$ where $x \sim y$ if and only if $x$ and $y$ share the same open neighbourhoods. In this section, we provide an explicit description of the Kolmogorov quotient of $\Zi \setminus \{0\}$ by utilising the set of finite subsets of prime associate classes.

Let $\widetilde{\Pe}$ be the set of associate classes of Gaussian primes. For a prime $\pi$, we denote its associate class by $[\pi]$. Let $X = \mathrm{Fin}(\widetilde{\Pe})$ be the collection of all finite subsets of $\widetilde{\Pe}$. For any finite set $F \in X$, we define the set
\[
\mathcal{O}_F := \{S \in X : S \cap F = \emptyset\}.
\]
Let $\mathcal{T}_X$ be the topology on $X$ generated by the collection $\mathcal{C} = \{\mathcal{O}_F : F \in X\}$.
\begin{lemma}\label{prop:basis-quotient}
The collection $\mathcal{C} = \{\mathcal{O}_{F} : F \in X\}$ constitutes a basis for the topology $\mathcal{T}_{X}$ on $X$.
\end{lemma}

\begin{proof}
First, observe that $X = \mathcal{O}_{\emptyset} \in \mathcal{C}$, ensuring that the collection covers the space. To establish the intersection property, let $\mathcal{O}_{F_{1}}, \mathcal{O}_{F_{2}} \in \mathcal{C}$ for some finite $F_{1}, F_{2} \subseteq \mathcal{P}_{e}$. A set $S \in X$ satisfies $S \in \mathcal{O}_{F_{1}} \cap \mathcal{O}_{F_{2}}$ if and only if $S \cap F_{1} = \emptyset$ and $S \cap F_{2} = \emptyset$, which is equivalent to $S \cap (F_{1} \cup F_{2}) = \emptyset$. Since $F_{1} \cup F_{2}$ is a finite subset of $\widetilde{\Pe}$, it follows that
\[
\mathcal{O}_{F_{1}} \cap \mathcal{O}_{F_{2}} = \mathcal{O}_{F_{1} \cup F_{2}} \in \mathcal{C}.
\]
As the collection is closed under finite intersections and covers $X$, it constitutes a basis.
\end{proof}
\begin{proposition}\label{prop:X-T0}
The topological space $(X, \mathcal{T}_X)$ is a Kolmogorov space; that is, it satisfies the $T_0$ separation axiom.
\end{proposition}

\begin{proof}
To establish that $(X, \mathcal{T}_X)$ is a Kolmogorov space, let $S$ and $T$ be distinct elements of $X$. Since $S, T$ are distinct sets, their symmetric difference $S\triangle T$ is necessarily nonempty. Without loss of generality, we may choose an associate class $[\pi]$ of a Gaussian prime $\pi$ such that $[\pi] \in S \setminus T$.

Consider the basis open set $\mathcal{O}_{\{[\pi]\}} = \{ A \in X : A \cap \{[\pi]\} = \emptyset \}$. By our choice of $[\pi]$, the following conditions hold:
\begin{itemize}
    \item $T \in \mathcal{O}_{\{[\pi]\}}$ since $[\pi] \notin T$;
    \item $S \notin \mathcal{O}_{\{[\pi]\}}$ since $[\pi] \in S$.
\end{itemize}
Thus, there exists an open neighbourhood of $T$ that does not contain $S$, which implies that the space $(X, \mathcal{T}_{X})$ is a Kolmogorov space.
\end{proof}

We define the support map $q : \mathbb{Z}[i] \setminus \{0\} \to X$ by $x\mapsto \Supp(x)$. The Theorem~\ref{kolmogorovquotient} establishes that the space $(X, \mathcal{T}_X)$ constitutes a Kolmogorov quotient of $\mathbb{Z}[i] \setminus \{0\}$ under the map $q$. The following Lemma~\ref{T0-quotient} establishes that $q$ is a quotient map.
\begin{lemma}\label{T0-quotient}
The map $q$ defined above is a surjective, continuous open map. Hence, $q$ is a quotient map. Moreover, $q$ identifies the topologically indistinguishable points.
\end{lemma}

\begin{proof}
 In order to verify the surjectivity of $q$, we consider an arbitrary element $S = \{[\pi_1], \dots, [\pi_n]\} \in X$ consisting of associate classes of Gaussian primes $\pi_{1},\dots,\pi_{n}$. Consider the product $x = \prod_{j=1}^n \pi_j$, we obtain an element in $\mathbb{Z}[i] \setminus \{0\}$ such that $\Supp(x) = S$. If $S = \emptyset$, we set $x = 1$, yielding $\Supp(1) = \emptyset$. Thus, $q$ is surjective.

In order to verify continuity of the map $q$ it is sufficient to verify the preimage of the basis elements $\mathcal{O}_F$ for $F \in X$ are open in $\Rbullet$. Fix $F = \{[\pi_1], \dots, [\pi_n]\}$ and let $\alpha = \prod_{j=1}^n \pi_j$ (with $\alpha = 1$ if $F = \emptyset$). The preimage is given by,
\[
q^{-1}(\mathcal{O}_F) = \{x \in \mathbb{Z}[i] \setminus \{0\} : \Supp(x) \cap F = \emptyset\} = \sigma_\alpha,
\]
which is a basis open set in the coprimality space of $\Zi$ . Thus, $q$ is a continuous map.

To establish that $q$ is a quotient map, we indeed show that it is an open map. Let $\sigma_\alpha$ be a basic open set in $\mathbb{Z}[i] \setminus \{0\}$. Its image under $q$ is,
\[
q(\sigma_\alpha) = \{ \Supp(x) : \Supp(x) \cap \Supp(\alpha) = \emptyset \} = \{ S \in X : S \cap \Supp(\alpha) = \emptyset \} = \mathcal{O}_{\Supp(\alpha)}.
\]
As $q$ maps a basis element of $\mathbb{Z}[i] \setminus \{0\}$ to the basis element of $X$, it is an open map. Since $q$ is a surjective, continuous, and open map, it follows that $q$ is a quotient map.

By Theorem~\ref{thm:indistinguishable},   $q(x)=q(y)$ if and only if $x$ and $y$ are topologically indistinguishable points. Thus the map $q$ identifies topologically indistinguishable points.
\end{proof}

We denote the Kolmogorov quotient of the coprimality space of $\Zi$ by $\widetilde{\mathbb{Z}[i]\setminus\{0\}}$. The following Theorem~\ref{kolmogorovquotient} describes $\widetilde{\Rbullet}$.

\begin{theorem}\label{kolmogorovquotient}
   The Kolmogorov quotient $\widetilde{\Rbullet}$ of $\Rbullet$ is homemorphic to the Kolmogorov space $(X,\mathcal{T}_{X})$. 
\end{theorem}

\begin{proof}
    In view of Lemma~\ref{T0-quotient} and the universal property of quotient topology, the support map $q$ induces a unique homeomorphism 
\[
\widetilde{q} : \widetilde{\mathbb{Z}[i]\setminus\{0\}} \to X
\]
which maps the equivalence class $[x]$ of a non-zero Gaussian integer $x$ to its prime support $\Supp(x)$.
\end{proof}
\section{Density and Infinitude of Gaussian Primes}

In this section, we establish the density of Gaussian primes in the coprimality topology and provide a topological proof of their infinitude. The following Theorem~\ref{thm:dense} establishes that the set of Gaussian primes is dense in the coprimality space of $\Zi$. Recall that $\mathcal{P}$ denotes the set of Gaussian primes. 
\begin{theorem}\label{thm:dense}
The set $\mathcal{P}$ is dense in $(\Rbullet,\mathcal{T})$.
Equivalently, for every $\alpha\in\Rbullet$, the basic open set $\sigma_\alpha$ contains
a Gaussian prime.
\end{theorem}

\begin{proof}
 It is enough to show that every nonempty basic open set $\sigma_{\alpha}$ contains at least one Gaussian prime. Choose a non-zero Gaussian integer $\alpha$.

If $\alpha$ is a unit, then $\sigma_{\alpha} = \mathbb{Z}[i] \setminus \{0\}$. Since $\mathbb{Z}[i] \setminus \{0\}$ clearly contains Gaussian primes (for instance, $1+i$), the intersection $\sigma_{\alpha} \cap \mathcal{P}$ is nonempty. We therefore assume that $\alpha \notin \mathcal{U}$.

Now consider the element $\gamma = \alpha + 1$. Clearly, $\gamma$ is non-zero Gaussian integer. Now, we consider two cases,

\vspace{1mm}
\emph{Case 1:} If $\gamma \notin \mathcal{U}$, then $\gamma$ necessarily possesses a Gaussian prime divisor $\pi$, as $\gamma$ is a non-zero non-unit. If $\pi$ divides $\alpha$, then $\pi$ would necessarily divide the difference $\gamma - \alpha = 1$, which is impossible as $\pi$ is a prime. Consequently, $\gcd(\pi, \alpha) \sim 1$, implying that $\pi \in \sigma_{\alpha} \cap \mathcal{P}$.

\vspace{1mm}
\emph{Case 2:} $\gamma \in \mathcal{U}$. Given that $\alpha \notin \mathcal{U}$ and $\alpha \neq 0$, the condition $\gamma \in \{1, -1, i, -i\}$ implies $\alpha\in\{-2, i-1, -1-i\}$. For each candidate in this set, we consider the prime $3 \in \mathbb{Z}$. It is straightforward to verify that the integer $3$ remains prime in $\mathbb{Z}[i]$. A direct computation shows,
\begin{itemize}
    \item If $\alpha = -2$, then $N(\alpha) = 4$. Since $\gcd(9, 4) = 1$ in $\mathbb{Z}$, $\gcd(3, -2) \sim 1$ in $\mathbb{Z}[i]$.
    \item If $\alpha = i-1$ , then $N(\alpha) = 2$. Since $\gcd(9, 2) = 1$ in $\mathbb{Z}$, $\gcd(3, i-1) \sim 1$ in $\mathbb{Z}[i]$.
    \item If $\alpha = -1-i$, then $N(\alpha) = 2$. Similarly, $\gcd(3, -1-i) \sim 1$ in $\mathbb{Z}[i]$.
\end{itemize}
In all subcases, $3 \in \sigma_{\alpha} \cap \mathcal{P}$. Consequently, every basic open set $\sigma_{\alpha}$ contains at least one Gaussian prime; that is to say, the set $\mathcal{P}$ is dense in $(\mathbb{Z}[i] \setminus \{0\}, \mathcal{T})$.
\end{proof}

\begin{corollary}
There are infinitely many Gaussian primes.
\end{corollary}

\begin{proof}
Suppose that the set $\mathcal{P}$ of Gaussian primes is finite. It follows that the collection of associate classes $\widetilde{\Pe}$ is finite as well; let these classes be denoted by $\{[\pi_1], \dots, [\pi_n]\}$. Let $\alpha = \prod_{j=1}^n \pi_j$ be the product of representatives from each class.
By construction, the basic open set $\sigma_\alpha$ is non-empty as it contains the unit $1$. 

Due to the Theorem~\ref{thm:dense} $\sigma_\alpha$ must contain at least one Gaussian prime $\pi$.
However, $\pi\in\sigma_\alpha$ implies that $\gcd(\pi,\alpha)\sim 1$. This yields a contradiction, as every Gaussian prime $\pi$ must necessarily divide $\alpha$ under the initial assumption of finiteness. Consequently, we conclude that the set $\mathcal{P}$ of Gaussian primes is infinite.
\end{proof}
\section{Conclusion}

This paper demonstrates that the coprimality topology on $\mathbb{Z}[i] \setminus \{0\}$ turns out to be a rigorous topological framework for characterising the arithmetic properties of the Gaussian integers. We have shown that topological indistinguishability in this space is determined precisely by the equality of Gaussian prime supports, which results in the collapse of the unit group $\mathcal{U}$ into a single indistinguishability class. Consequently, the space fails to satisfy the $T_0$ separation axiom and satisfies strong connectivity properties due to the presence of units in every nonempty open set.

Furthermore, we provided an explicit description of the Kolmogorov quotient of the coprimality space of $\Zi$. Finally, our density argument establishes that the set of Gaussian primes $\mathcal{P}$ is dense in $(\mathbb{Z}[i] \setminus \{0\}, \mathcal{T})$. This result provides an alternative topological proof for the infinitude of Gaussian primes while highlighting the necessity of considering the ring's specific unit structure, particularly in the analysis of the element $\alpha+1$ in Theorem~\ref{thm:dense}.

A natural direction for future research is to extend this framework to other unique factorization domains with different arithmetic properties. Consider, for instance, the ring $\mathbb{Z}[\sqrt{2}]=\{a+b\sqrt{2} \mid a, b\in\mathbb{Z}\}$. Like $\Zi$, it is a Euclidean domain; however, unlike $\Zi$, its group of units is infinite. This structural difference presents a non-trivial challenge for the density argument. In this regard, we pose the following problem:

\begin{problem}
Is the set of primes dense in $\mathbb{Z}[\sqrt{2}]\setminus\{0\}$ endowed with the coprimality topology?
\end{problem}
\subsection*{Acknowledgment}
The author is grateful for financial support in the form of Prime Minister’s Research Fellowship, Government of India (PMRF/2502403).

\bibliographystyle{plain} 
\bibliography{references}
\end{document}